\documentclass{aims}
\usepackage{amsmath}
  \usepackage{paralist}
  \usepackage{graphics} %% add this and next lines if pictures should be in esp format
  \usepackage{epsfig} %For pictures: screened artwork should be set up with an 85 or 100 line screen
\usepackage{graphicx}  \usepackage{epstopdf}%This is to transfer .eps figure to .pdf figure; please compile your paper using PDFLeTex or PDFTeXify.
 \usepackage[colorlinks=true]{hyperref}
\usepackage[utf8]{inputenc}
\usepackage[T1]{fontenc}
\usepackage{textcomp}
\usepackage{gensymb}
\hypersetup{urlcolor=blue, citecolor=red}

\def\currentvolume{X}
 \def\currentissue{X}
  \def\currentyear{200X}
   \def\currentmonth{XX}
    \def\ppages{X--XX}
     \def\DOI{10.3934/xx.xx.xx.xx}

\newtheorem{theorem}{Theorem}[section]

\newtheorem{lemma}[theorem]{Lemma}
\newtheorem{proposition}{Proposition}

\theoremstyle{definition}
\newtheorem{definition}[theorem]{Definition}
\newtheorem{remark}{Remark}

\title[Transport equation and zero quadratic..] %Use the shortened version of the full title
      {The transport equation and zero quadratic variation processes}

\author[Jorge Clarke, Christian Olivera and Ciprian Tudor]{}


 \email{jorge.clarke@usm.cl}
 \email{colivera@ime.unicamp.br}
 \email{tudor@math.univ-lille1.fr}

\thanks{The first author is supported by FONDECYT grant N\textdegree 3150506.} 
\thanks{The second author is supported by FAPESP 2015/04723-2   and CNPq through the grant 460713/2014-0.}
\thanks{The third author is supported by the CNCS grant PN-II-ID-PCCE-2011-2-0015.}
\thanks{Dedicated to Bj\"orn Schmalfuss on  occasion of his 60th birthday.}

\thanks{$^*$ Corresponding author: Jorge Clarke}

\begin{document}
\maketitle

% Enter the first author's name and address:
\centerline{\scshape Jorge Clarke $^*$}
\medskip
{\footnotesize
% please put the address of the first author
 \centerline{Departamento de Matem\'atica,
 Universidad T\'ecnica Federico Santa Mar\'ia.}
 \centerline{Avda. Espa\~na 1680, Valpara\'iso, Chile.}
   \centerline{Laboratoire Paul Painlev\'e, Universit\'e de Lille 1. F-59655 Villeneuve d'Ascq, France.}
   %\centerline{ Springfield, MO 65801-2604, USA}
} % Do not forget to end the {\footnotesize by the sign }

\medskip

\centerline{\scshape Christian Olivera}
\medskip
{\footnotesize
 % please put the address of the second  and third author
 \centerline{Departamento de Matem\'atica, Universidade Estadual de Campinas.}
 \centerline{13.081-970-Campinas-SP-Brazil.}
   %\centerline{Other lines}
   %\centerline{Springfield, MO 65810, USA}
}

\medskip

\centerline{\scshape Ciprian Tudor}
\medskip
{\footnotesize
 % please put the address of the second  and third author
 \centerline{Laboratoire Paul Painlev\'e, Universit\'e de Lille 1. F-59655 Villeneuve d'Ascq, France.}
   \centerline{Academy of Economical Studies, Bucharest, Romania.}
   %\centerline{Springfield, MO 65810, USA}
}

\bigskip

% The name of the associate editor will be entered by an editorial staff
% "Communicated by the associate editor name" is not needed for special issue.
% \centerline{(Communicated by the associate editor name)}

%The abstract of your paper
\begin{abstract}
We analyze the transport equation driven by a zero quadratic variation process. Using the stochastic calculus via regularization and the Malliavin calculus techniques, we prove the existence, uniqueness and absolute continuity of the law of the solution. As an example, we discuss the case when the noise is a Hermite process.
\end{abstract}

%The title of your section 1

\section{Introduction}

Transport phenomena arise in many research fields; geosciences, physics, biology, even in social sciences, for naming just a few. The linear transport equation,
 
\begin{equation}\label{trasports}
    \partial_t u(t, x) +  b(t,x) \cdot  \nabla u(t,x)  = 0 \, ,
\end{equation}
emerges as a model for the concentration (density) of a pollutant in a flow, and may be considered as a particular case of the convection/advection equation when the flow under consideration is incompressible (i.e. has zero divergence). 

To point out some applications of this equation, we refer the reader to the works of Lions (\cite{lion1,lion2}) for a detailed exposition of its use in fluid dynamics, the work of Dafermos \cite{Dafermos} for its relation to conservation laws, and the work of Perthame \cite{Perhame} to understand the relevance of equations like (\ref{trasports}) in biology. Transport equations also appear in models for ocean salinity, see \cite{DGX}.

In this paper we analyze a stochastic transport equation. The following one-dimensional Cauchy problem is considered: given an initial-data $u_0$,
find $u(t,x;\omega) \in \mathbb{R}$, satisfying 
\begin{equation}\label{transport}
 \left \{
\begin{split}
    & \partial_t u(t, x;\omega) + \partial_x u(t, x;\omega)  \, \Big(\, b(t, x) + \frac{d Z_{t}}{dt}(\omega) \Big)= 0,
    \\ 
    & u(t_{0},x)= u_{0}(x),
\end{split}
\right .
\end{equation}
with $(t,x) \in U_T = [0,T] \times \mathbb{R}, \ \omega \in \Omega $, $b:[0,T] \times \mathbb{R}\to \mathbb{R}$ a given vector field, and the noise $(Z_{t})_{t\geq 0} $ is a stochastic process with zero quadratic variation. Problem (\ref{transport}) may be understood as a model for the concentration (density) of a pollutant in a flow where the velocity field has a random perturbation.

The stochastic transport equation driven by the standard Brownian motion was first addressed in Kunita's books (see \cite{Ku3}, \cite{Ku2}). More recently it has been studied by several authors; in \cite{CO} the linear additive case is considered, existence and uniqueness of weak $L^{p}$-solutions and a representation for the general solution were shown. The non-blow-up problem is addressed for the multiplicative case with Stratonovich form in \cite{Fre1}. In \cite{FGP2} the authors have shown that the introduction of a multiplicative noise in the PDE allows some improvements in the traditional hypothesis needed to prove that the problem is well-posed, this is extended later to a non-linear case in \cite{Oli}. A new uniqueness result is obtained in \cite{Maurelli} by means of Wiener-chaos decomposition, and working on the associated Kolmogorov equation.  
The extension of the model to the fractional Brownian noise has been done in \cite{OT}, where the existence of density of the solution and Gaussian estimates of the density were proven.

Our purpose is to solve the equation (\ref{transport}) and to analyze the properties of its solution in the case when the noise is a more general stochastic process, possibly non-Gaussian. We will focus on the situation when the noise $Z$ in (\ref{transport}) is a stochastic process with zero quadratic variation, this is well defined in the next section of the paper.

The reason why we chose such a noise is that the stochastic integration theory in the sense of Russo-Vallois (see \cite{RV}, \cite{RV1}) can be applied to it. In fact, the stochastic integral in (\ref{transport}) will be understood  as a symmetric integral in the Russo-Vallois sense with respect to the noise $Z$. Besides, in most of the papers cited in the previous paragraph, the It\^o-Wentzell formula plays a crucial role in the characterization of the solution, consider a zero quadratic variation process is as far as one can go in order to prove that characterization, avoiding the presence of second order terms (see Proposition 9 in \cite{FlandRusso}). Among the zero covariance processes lies the fractional Brownian motion (for $H \geq 1/2$), a self-similar process that find some of their applications in various kind of phenomena, going from hydrology and surface modelling to network traffic analysis and mathematical finance,
to name a few.

In this paper, first, the existence and uniqueness of the solution to (\ref{transport}) is proved, by using the so-called method of characteristics that comes from the works of Kunita (\cite{Ku}, \cite{Ku3}). Then it is show that the solution can be expressed as the initial data $u_{0}$ applied to the inverse flow associated to the transport equation. Later, using the techniques of Malliavin calculus on this representation, it is possible to prove the existence of the density of this solution. This will be done by showing that the Malliavin derivative of the solution is strictly positive. To this end, the explicit  expression of the Malliavin derivative must be calculated and controlled.

The outline of this paper is as follows: In Section 2 we present the basics definitions of the integration theory in the sense of Russo-Vallois, stochastic calculus via regularization, and the Malliavin techniques. The existence and uniqueness of weak solution to problem (\ref{transport}) is proved in Section 3. Section 4 provides the demonstration of the existence of the density of the solution. Finally in Section 5 we  illustrate the results considering a Hermite process as the driving noise.  

\section{Preliminaries}

In this section, we recall the notions from the stochastic calculus via regularization and from Malliavin calculus that we will use in what follows. More details can be found in \cite{RV1} or \cite{N}. 

\subsection{Stochastic calculus via regularization}

Throughout this paper $T$ will be a fixed positive real number. We recall the definition of the symmetric integral $d^{\circ} X$ that will appear in  (\ref{transport}) and in the definition of the solution in formula (\ref{DISTINTSTR}). 

Assume $(X_t)_{ t\geq 0} $ is  a continuous process and
$(Y_t)_{ t\geq 0}$ is  a process with paths in
$L_{loc}^{1}(\mathbb{R}^{+})$, i.e. for any $ b > 0$, $
\int_{0}^{b}|Y_t| dt <\infty$ a.s. The generalized stochastic integrals (forward, backward and symmetric)  are defined through a regularization procedure see \cite{RV}, \cite{RV1}. Here we recall  only the definition of the  symmetric integral (actually, since we are dealing with zero quadratic variation processes,  the three integrals will coincide) . Let $ I^{0}(\varepsilon, Y, dX)$  be the $\varepsilon-$symmetric integral

$$
I^{0}(\varepsilon, Y, dX)(t) = \int_{0}^{t} Y_{s}
\frac{(X_{s+\varepsilon}-X_{s-\varepsilon})}{2\varepsilon} ds, \ \ \  t \geq 0.
$$
The symmetric integral $\int_{0}^{t} Y _{s}d^{\circ} X_{s}$ is defined as 

\begin{equation}\label{simint}
   \int_{0}^{t} Y_{s} d^{\circ} X_{s}: =
\lim_{\varepsilon\rightarrow 0}I^{0}(\varepsilon, Y, dX)(t),
\end{equation}
for every $t\in [0,T]$, provided the limit exists in the {\it ucp} sense (uniformly on compacts in probability). 

In a similar way, the covariation or generalized bracket, $[X,Y]_{t}$ of two stochastic processes $X$ and $Y$ is defined as the limit {\it ucp} when $\varepsilon$ goes to zero of 

\begin{equation*}
[X, Y] _{\varepsilon, t}= \frac{1}{\varepsilon} \int_{0} ^{t} \left( X_{s+\varepsilon} -X_{s}\right)
\left( Y_{s+\varepsilon} -Y_{s}\right)ds, \ \ \ t\geq 0.
\end{equation*} 
Note that $[X,Y]$ coincide with the classical bracket when $X$ and $Y$ are semi-martingales.

 A process $X$, such that $[X,X]$ exists, is called finite quadratic variation processes. If $[X,X] \equiv 0$ we say that $X$ is a zero quadratic variation process.  Our integrand in (\ref{transport}) will be such a zero quadratic variation process.

\subsection{Malliavin derivative}

 We also present the elements from the Malliavin calculus that will be used in the paper.  We refer to \cite{N} for a more complete exposition. Consider ${\mathcal{H}}$ a real separable Hilbert space and $(B (\varphi), \varphi\in{\mathcal{H}})$ an isonormal Gaussian process \index{Gaussian process} on a probability space $(\Omega, {\mathcal{A}}, \mathbb{P})$, that is, a centered Gaussian family of random variables such that $\mathbb{E}\left( B(\varphi) B(\psi) \right)  = \langle\varphi, \psi\rangle_{{\mathcal{H}}}$.

We denote by $D$  the Malliavin  derivative operator that acts on smooth functions of the form $F=g(B(\varphi _{1}), \ldots , B(\varphi_{n}))$ ($g$ is a smooth function with compact support and $\varphi_{i} \in {{\mathcal{H}}}, i=1,...,n$)
\begin{equation*}
DF=\sum_{i=1}^{n}\frac{\partial g}{\partial x_{i}}(B(\varphi _{1}), \ldots , B(\varphi_{n}))\varphi_{i}.
\end{equation*}

It can be checked that the operator $D$ is closable from $\mathcal{S}$ (the space of smooth functionals as above) into $ L^{2}(\Omega; \mathcal{H})$ and it can be extended to the space $\mathbb{D} ^{1,p}$ which is the closure of $\mathcal{S}$ with respect to the norm
\begin{equation*}
\Vert F\Vert _{1,p} ^{p} = \mathbb{E} F ^{p} + \mathbb{E} \Vert DF\Vert _{\mathcal{H}} ^{p}. 
\end{equation*}

We denote by  $ \mathbb{D} ^{k, \infty}:= \cap _{p\geq } \mathbb{D} ^{k,p}$ for every $k\geq 1$. In this paper, $\mathcal{H}$ will be the standard Hilbert space $L^ {2}([0,T])$.

We will make use of the chain rule for the Malliavin derivative (see Proposition 1.2.4 in \cite{N}). That is, if $\varphi: \mathbb{R}\to \mathbb{R}$ is a differentiable with bounded derivative and $F\in \mathbb{D} ^ {1,2}$, then  $\varphi (F) \in \mathbb{D} ^ {1,2}$ and
\begin{equation}
\label{chain}
D\varphi(F)= \varphi ' (F) DF.
\end{equation}

An important role of the Malliavin calculus  is that it provides criteria for the existence of the density of a random variable. Here we will use the following result: if $F$ is a random variable in $\mathbb{D} ^ {1,2}$ such that $\Vert DF \Vert _{\mathcal{H}} >0$ almost surely, then $F$ admits a density with respect to the Lebesgue measure (see e.g. Theorem 2.1.3 in \cite{N}).

\section{Existence and uniqueness}

 This section presents the results concerning the existence of the weak solution for the stochastic transport equation (\ref{transport}) driven by a zero quadratic variation process.

Let $(\Omega, \mathcal{F}, P)$ be a fixed probability space and $(W_{t})_{t\in [0,T]} $ a standard Wiener process on it. Consider a continuous process $(Z_{t})_{t\geq 0}$, adapted to the filtration generated by $W$, and such that 
$$[Z, Z]_{t}=0, \hskip0.5cm \forall \  t\geq 0. $$ 
The quadratic variation $[Z,Z]$ is understood in the sense of stochastic calculus via regularization, as introduced in Section 2.1.

We consider the one-dimensional Cauchy problem (\ref{transport}) driven by the process $Z$, with a given initial-data $u_0$. 

We recall the notion of weak solution to (\ref{transport}) (see \cite{FGP2}).

\begin{definition}\label{defisolu}  A stochastic process
$u\in  L^{\infty}(\Omega\times[0, T]\times \mathbb{R})$ is called 
a weak $L^{p}-$solution of the Cauchy problem (\ref{transport}), if for any $\varphi \in C_c^{\infty}(\mathbb{R})$,  $\int _{\mathbb{R}} u(t,x)\varphi(x)dx$ is an    adapted real valued process which has a continuous modification, finite covariation, and for all $t \in [0,T]$, $\mathbb{P}$-almost surely

\begin{equation}\label{DISTINTSTR}
\begin{split}
	\int_{\mathbb{R}} u(t,x) \varphi(x) dx =& \int_{\mathbb{R}}  u_{0}(x) \varphi(x) \ dx
   +\int_{0}^{t} \int_{\mathbb{R}}  u(s,x) \ b(s,x) \partial_{x} \varphi(x) \ dx ds\nonumber  
\\
+ &  \int_{0}^{t} \int_{\mathbb{R}} u(s,x)  b'(s,x)  \varphi(x) \ dx ds  + \int_{0}^{t} \int_{\mathbb{R}} u(s,x) \partial_{x} \varphi(x) \ dx  d^{\circ}Z _s,  
\end{split}
\end{equation}
where $b'(s,x)$ denotes the derivative of $b(s,x)$ with respect to the spatial variable $x$, and the integral $d^{\circ}Z$ is a symmetric integral defined via regularization (see (\ref{simint})). 
\end{definition}

Following the arguments presented in \cite{OT}, the existence and uniqueness of the weak solution to (\ref{transport}) follows immediately. 

\vskip0.2cm

\begin{proposition}
Assume that  $b \in L^{\infty}((0,T); C_b^{1}(\mathbb{R^{d}}))$. Then there exists a $C^{1} (\mathbb{R^{d}}) $ stochastic flow of diffeomorhism $\left(X_{s,t},\  0\leq s\leq t\leq T\right)$, that satisfies 
\begin{equation}
\label{1}
X_{s,t} (x)= x+ \int_{s}^{t} b(u, X_{s,u} (x)) du + Z_{t} - Z_{s}
\end{equation}
for every $x\in \mathbb{R}^{d}$. 
Moreover, given $u_{0}\in L^{\infty}(\mathbb{R})$, the stochastic process 
\begin{equation}
\label{19a-1}u(t, x):= u_0(X_{t}^{-1}(x)), \hskip0.5cm t\in [0,T], x\in \mathbb{R}
\end{equation} is the 
unique weak $L^{\infty}-$solution of the Cauchy problem \ref{transport}, where
$X_{t}:= X_{0,t}$ for every $t\in [0,T]$.
\end{proposition}

\begin{proof}
 The demonstration follows closely the lines of the proof of Theorem 1 in \cite{OT}, which is based on the It\^o formula in the Russo-Vallois sense  for functions depending on $\omega$. Although in \cite{OT} the noise is a fractional Brownian motion with Hurst parameter $H>\frac{1}{2}$, the only property of the fBm needed in the demonstration is the fact that, for $H>\frac{1}{2}$, is a zero quadratic variation process. Therefore, all the steps in the proof of Theorem 1 in \cite{OT} remain valid when the noise is a general zero quadratic variation process.
\end{proof}

\vskip0.3cm

From (\ref{19a-1}), it is clear that the properties of the solution to the transport equation (\ref{transport}), in particular the existence of its density, will depend on the initial condition $u_{0}$ and on the properties of the inverse flow $X_{t}^{-1}$. Since later in the paper we will assume on $u_{0}$ as much regularity as needed, we focus on the analysis of the inverse flow. Let's start by describing its dynamic.

\begin{lemma}
Assume $  b\in L ^{\infty} \left( (0,T); C_{b} ^{1} (\mathbb{R} )\right) \cap C \left( (0,T) \times \mathbb{R}\right).$ Then the inverse flow satisfies the backward stochastic equation

\begin{equation}\label{back}
Y_{s,t}(x)= x- \int_{s} ^{t} b(r,  Y_{r,t}(x))dr - (Z_{t}- Z_{s})
\end{equation}
for every $0\leq s\leq t\leq T$ and for every $x\in \mathbb{R}$. 

Moreover, $Y$ is the unique process that satisfies (\ref{back}) with $Y_{s,s}(x)=x$. 
\end{lemma}
\begin{proof} Analogously to the proof of Lemma 2 in \cite{OT}. 
\end{proof}

\begin{remark}
If we set
$R_{t,x}(u) = Y_{t-u, t}(x) $ for $u\in [0,t]$, and $x\in \mathbb{R}$, then we have
\begin{equation}\label{5m-4}
R_{t,x}(u)= x- \int_{0} ^{u} b(t-a, R_{t,x}(a) ) da - (Z_{t} - Z_{t-u}).
\end{equation} 
This can be easily seen from (\ref{back}) by making the change of variable $a=t-r.$ If we set
$$B_{t}(a,x)=- b(t-a, x) $$
for $a\leq t$, $ x\in \mathbb{R}$, and 
\begin{equation*}
Z_{u,t} = - (Z_{t}- Z_{t-u})
\end{equation*}
for $t\in [0,T]$ and $u\leq t$, then (\ref{5m-4}) becomes
\begin{equation}\label{5m-5}
R_{t,x}(u)=x + \int_{0} ^{u} B_{t} (a, R_{t,x}(a))da + Z_{u,t}.
\end{equation}
\end{remark}

We will actually use the above equation in order to obtain the properties of the inverse flow.

\section{The Malliavin derivative and the density of the solution}

 We will show that the solution to (\ref{transport}) is Malliavin differentiable and, using the techniques of the Malliavin calculus, that it admits a density with respect to the Lebesgue measure. From the representation (\ref{19a-1}), it is enough to focus on the Malliavin derivative of the inverse flow $X_{t} ^ {-1}$ whose dynamic is governed by (\ref{back}) or (\ref{5m-5}).

\subsection{Malliavin differentiability of the inverse flow}

 Throughout this section we assume that  

\begin{equation}\label{6m-1}
b\in  L ^{\infty} \left( (0,T); C_{b} ^{1} (\mathbb{R} )\right) \cap C \left( (0,T) \times \mathbb{R}\right).
\end{equation}
The noise $Z$  is a zero quadratic variation process, adapted to the filtration $(\mathcal{F}_{t})_{t\in [0,T]}$ such that $Z_{t}\in \mathbb{D} ^ {1,2}$ for every $t\in [0,T]$. We suppose also that

\begin{equation}
\label{6m-2}
\sup_{t\in [0,T]}\mathbb{E} \left| Z_{t}\right| ^{2} < \infty  \mbox{ and } \sup_{t\in [0,T]}\mathbb{E} \Vert DZ_{t} \Vert ^{2} _{L^{2}([0,T ])}<\infty.
\end{equation}

\begin{lemma}

Under hypothesis (\ref{6m-1}) and (\ref{6m-2}),  the equation (\ref{5m-5}) has an unique solution.

\end{lemma}
\begin{proof}
	Using classic Picard iterations is clear that there exists an unique solution to equation (\ref{5m-5}). We refer the reader to the proofs of Lemma 5 in \cite{FeRo} or Lemma 2.2.1 in \cite{N} for similar results.
\end{proof}

Now, the goal is  to show the Malliavin differentiability of the solution to (\ref{5m-5}).

\begin{proposition}\label{p2} Assume (\ref{6m-1}) and (\ref{6m-2}). Then, for $x\in \mathbb{R}, t\in [0,T], u\leq t$, the random variable $R_{t,x} (u)$ given by (\ref{5m-5}) belongs to $ \mathbb{D} ^{1,2}$. Moreover, the Malliavin derivative of $R_{t,x}(u)$ satisfies
\begin{equation}
\label{eqD}
D_{\alpha}R_{t,x}(u)= \int_{0}^{u} B'_{t}(s, R_{t,x}(s)) D_{\alpha} R_{t,x} (s)ds + D_{\alpha}Z_{u,t}
\end{equation}
for every $\alpha <t$. 
\end{proposition}
\begin{proof}
	Fix $t\in [0,T], x\in \mathbb{R}$. Define the usual iterations 
$$R ^{(0)} _{t,x}(u)=x \mbox{ for every  } t\in [0,T]$$
and for $n\geq 0$,
$$ R^{(n+1)} _{t,x}(u)= x + \int_{0}^{u} B_{t}(s, R ^{(n)} _{t,x})(s) ds + Z _{u,t} $$
with $0\leq u\leq t.$

It is clear that $R ^{(0)}_{t,x}(u)$ is $\mathcal{F}_{t}$-measurable and belongs to $\mathbb{D} ^{1,2}$. By a trivial induction argument and using Proposition 1.2.4 in \cite{N}, we have that $R^{(n)}_{t,x} (u)$ is $ \mathcal{F}_{t}$-measurable and $ R^{(n)}_{t} $  belongs to $\mathbb{D} ^{1,2}$ for every $n\geq 0$. This implies that
$$D_{\alpha} R^ {(n)} _{t,x}(u)=0 \mbox{ if } \alpha >t.$$

Our proof follows the following standard arguments: first, we notice the $L^ {2}$ convergence of $R ^ {(n)}_{t,x}(u)$ to $ R_{t,x}(u)$. Secondly, we prove an uniform bound on the sequence of  Malliavin derivatives of $R ^ {(n)}_{t,x}(u)$ and then we conclude the Malliavin differentiability of $R_{t,x}(u)$ along with the expression of its Malliavin derivative. 

Concerning the first step, we only  remind that (see e.g. the proof of Lemma 5 in \cite{FeRo}), for every $u\in [0,T]$
\begin{equation}
\label{6m-3}
\sup_{t\geq u } \mathbb{E} \left| R ^{(n)}_{t,x}(u) - R_{t,x}(u) \right| ^{2} \to _{n\to \infty} 0.
\end{equation}

Let us use the notation
\begin{equation}
\label{cond-noise}
m_{T}:=\sup_{t\in [0,T] } \sup_{u\in [0,t]}\mathbb{E} \Vert D Z_{u, t} \Vert _{ L^{2} ([0,T])}^{2} 
\end{equation}
  which is a finite positive constant due to (\ref{6m-2}).
We now show that
\begin{equation}\label{7m-1}
\sup_{n\geq 1}\sup_{t\geq u} \mathbb{E} \Vert D R^{(n)}_{t,x}(u) \Vert _{ L^{2} ([0,T])}^{2} < \infty.
\end{equation}
For $u\in [0,T]$ fixed, denote by
$$g^{n}_{u}= \sup_{t\geq u} \mathbb{E} \Vert DR^{(n)} _{t,x}(u) \Vert ^{2} _{ L^{2}([0,T])} \ = \sup_{t\geq u} \mathbb{E} \int_{0} ^{T} (D_{\alpha} R^{(n)}_{t,x}(u)) ^{2} d\alpha.$$
Then, using the fact that $b'$ is bounded
\begin{equation*}
\begin{split}
g^{(n+1)} _{u} = & \sup_{t\geq u} \mathbb{E} \left( \int_{0}^{T} \int_{0}^{u} B_{t}'(a, R ^{(n)}_{t,x}(a) ) D_{\alpha } R^{(n)}_{t,x} (a) da + D_{\alpha } Z_{u,t} \right) ^{2} d\alpha \\
 \leq  & C(T) \sup_{t\geq u} \left( \mathbb{E} \int_{0} ^{T} d\alpha \int_{0} ^{u} da( D_{\alpha } R^{(n)}_{t,x} (a)) ^{2} + m_{T}\right) \\
 \leq  & C(T) \sup_{t\geq u} \left( \int_{0}^{u} da \sup_{t\geq a} \mathbb{E} \int_{0} ^{T} d\alpha ( D_{\alpha } R^{(n)}_{t,x} (a)) ^{2} + m_{T}\right) \\
 = & C(T) ( \int_{0} ^{u} g^{n}_{a} da + m_{T} ),
\end{split}
\end{equation*}
because $a\leq u \leq t$. By taking the supremum over $n\geq 1$, and then applying Gronwall lemma, we get 
$$\sup_{n\geq 1} g ^{n} _{u} \leq C_{1}(T) e ^{C_{2}(T)} <\infty.$$ 
So, (\ref{7m-1}) holds. This, together with (\ref{6m-3}) and Lemma 1.2.3 in \cite{N} implies that 
$$R_{t,x}(u) \in \mathbb{D} ^{1,2} \mbox{ for every } t\in [0,T],$$
and the sequence of derivatives $(DR_{t,x} ^{(n)}(u))_{n\geq 0}$ is convergent in $ L^{2} (\Omega \times [0,T])$ to $DR_{t,x}(u)$, hence (\ref{eqD}) holds true.
\end{proof}

\vskip0.5cm

\subsection{The density of the solution}

 Let $Y_{s,t}(x)$ be the inverse flow given by (\ref{back}), with $0\leq s\leq t\leq T$ and $x\in \mathbb{R}$. Recall the notation $R_{t,x}(u) = Y_{t-u, t}(x)$ if $u< t$, and the hypothesis on the noise $Z$.

% in (\ref{transport})  is a zero quadratic variation process satisfying (\ref{6m-2}).

From (\ref{5m-5}) and Proposition \ref{p2}, it follows that $Y_{s,t}(x) $ is Malliavin differentiable. Our next step is to find the expression of its Malliavin derivative.

\begin{proposition}
Assume that $b$ satisfies (\ref{6m-1}), $Z$ satisfies (\ref{6m-2}) and let $Y$ be given by (\ref{back}). Then for every $s\leq t, \alpha \leq t $ and $x\in \mathbb{R}$, we have  

\begin{equation}\label{dy}
\begin{split}
D_{\alpha} Y _{s,t} (x) =& 1_{(0,T)}(\alpha) e ^{- \int_{s} ^{t} b'(u, Y_{u,t}(x) ) du } \int_{s} ^{t}  b'(u, Y_{u,t}(x)) D_{\alpha} (Z_{u,t}) e ^{\int_{u}^{t} b'(r, Y_{r,t}(x)) dr } \ du \nonumber \\
&  + \  D_{\alpha} Z_{s,t}. 
\end{split}
\end{equation}
%\\
%
%&\times & \left( \right) 
\end{proposition}
\begin{proof}
	First we notice that, if $\alpha >t$, $D_{\alpha} Y_{s,t}(x) =0$, because $Y_{s,t}(x)$ is $\mathcal{F}_{t}$-measurable. Now, let's observe that 
%By the method of variations of constants, we get, for $u<T$ and $\alpha <t$

\begin{equation*}
D_{\alpha}R_{t,x}(s)= \int_{0}^{s} B'_{t}(u, R_{t,x}(u)) D_{\alpha} R_{t,x} (u)du + D_{\alpha}Z_{s,t},
\end{equation*}
where we recall the notation $Z_{s,t} = -(Z_{t} - Z_{t-s})$ for $0\leq s \leq t \leq T$.
Considering the previous relation as an integral equation of functions depending on the variable $s$, regularizing $Z$ by standard mollifiers, applying the method of variations of parameters and then integration by parts, we obtain
%leads us to
%\begin{eqnarray*}
%D_{\alpha } R_{t,x} (s) &=& e ^{ \int_{0} ^{s} B'_{t}(u, R_{t,x} (u) ) du } \int_{0} ^{s} \frac{d}{du} ( D_{\alpha}(Z_{u,t})) e ^{- \int_{0} ^{u} B'_{t}(r, R_{t,x} (r) )dr } du.
%\end{eqnarray*}
%By regularization and integration by parts, we obtain
\begin{equation*}
\begin{split}
D_{\alpha } R_{t,x} (s) =& D_{\alpha}Z_{s,t} \\
 & + e ^{ \int_{0} ^{s} B'_{t}(u, R_{t,x} (u) ) du } \int_{0} ^{s} B'_{t}(u, R_{t,x} (u) ) D_{\alpha}(Z_{u,t}) e ^{- \int_{0} ^{u} B'_{t}(r, R_{t,x} (r) )dr } du  
\end{split}
\end{equation*}
Writting this in terms of the inverse flow $Y$, 
\begin{equation*}
\begin{split}
D_{\alpha} Y_{t-s,t} (x) =& D_{\alpha}Z_{s,t} 
 + e ^{- \int_{0} ^{s} b'(t-u, Y_{t-u,x} (u) ) du } \  \cdot 
 \\
 & \int_{0} ^{s} b'(t-u,Y_{t-u,x} (u) ) D_{\alpha}(Z_{u,t}) e ^{\int_{0} ^{u} b'(t-r, Y_{t-r,t} (r) )dr } du.
\end{split}
\end{equation*}
Using the changes of variables $u' = t-u$, $r' = t-r$  and the notation $s' = t-s$  we obtain the desired result.
\end{proof}

Let us prove that the random variable $Y_{s,t}$ admits a density, for every $0\leq s\leq t\leq T$.

\begin{proposition}
Fix $0\leq s\leq t\leq T$. Assume (\ref{6m-1}) and (\ref{6m-2}). In addition we will suppose that for every $0<s<t\leq T$
\begin{equation}
\label{7m-3}
 \Vert D Z_{t} \Vert  ^ {2}_{ L^{2} ([t-s, T])} = \int_{t-s} ^{T} (D_{\alpha} Z_{t} ) ^{2} d\alpha >0, \mbox{ almost surely. }
\end{equation}
Then the law of $Y_{s,t}(x)$ is absolutely continuous with respect to the Lebesgue measure.

\end{proposition}
\begin{proof}
	By Theorem 2.1.3 in \cite{N},  we need to prove that

$$ \Vert D Y_{s,t} \Vert ^ {2} _{ L^{2} ([0, T])}= \int_{0} ^{T} (D_{\alpha} Y _{s,t} ) ^{2} d\alpha >0 \mbox{ almost surely. }$$
We will use the inequality 
$$\int_{0} ^{T} (D_{\alpha} Y _{s,t} ) ^{2} d\alpha \geq \int_{t-s} ^{T} (D_{\alpha} Y _{s,t} ) ^{2} d\alpha $$
and the fact that, for $\alpha \in (t-s, T)$ and $u\in (s,t)$, we have 
$$D_{\alpha } Z_{s,t}= -D_{\alpha } Z_{t} \mbox { and } D_{\alpha } Z_{u,t}= -D_{\alpha } Z_{t}$$ since
$D_{\alpha } Z _{t-s} = D_{\alpha } Z _{t-u}=0.$
Therefore
\begin{equation*}
\begin{split}
& \int_{0} ^{T} (D_{\alpha} Y _{s,t} ) ^{2} d\alpha \  \geq \  \int_{t-s} ^{T} (D_{\alpha} Y _{s,t} ) ^{2} d\alpha
\\
=&\int_{t-s} ^{T} \left( - e ^{ -\int_{s} ^{t} b'(u, Y_{u,t}(x) ) du } \int_{s} ^{t}  b'(u, Y_{u,t}(x)) D_{\alpha} Z_{t} e ^{\int_{u}^{t} b'(r, Y_{r,t}(x)) dr } \ du - D_{\alpha} Z_{t}\right)^{2}d\alpha 
\\
=&\int_{t-s} ^{T} (D_{\alpha } Z_{t} ) ^{2} 
\left(  e ^{ -\int_{s} ^{t} b'(u, Y_{u,t}(x) ) du } \int_{s} ^{t}  b'(u, Y_{u,t}(x)) e ^{\int_{u}^{t} b'(r, Y_{r,t}(x)) dr } \ du +1\right)^{2}d\alpha 
\end{split}
\end{equation*}

We claim that
\begin{equation}
\label{25a-1}
 e ^{ -\int_{s} ^{t} b'(u, Y_{u,t}(x) ) du } \int_{s} ^{t}  b'(u, Y_{u,t}(x)) e ^{\int_{u}^{t} b'(r, Y_{r,t}(x)) dr } \ du +1\ > \ C \ > \ 0.
\end{equation}
Indeed, considering the boundedness of $b'$, and the inequalities
\begin{equation}\label{5m-3}
e ^{-(t-s) \Vert b'\Vert _{\infty} } \leq  e^{ \pm \int_{s} ^{t} b'(a, Y_{a, t}(x) ) da }\leq e ^{(t-s) \Vert b'\Vert _{\infty} },
\end{equation}
we may note that
\begin{eqnarray*}
& & e ^{ -\int_{s} ^{t} b'(u, Y_{u,t}(x) ) du } \int_{s} ^{t}  b'(u, Y_{u,t}(x)) e ^{\int_{u}^{t} b'(r, Y_{r,t}(x)) dr } \ du \\
&\geq & \inf_{u \in (s,t); x \in \mathbb{R}} b'(u,x) \cdot (t-s) \cdot e ^{ -2(t-s) \Vert b'\Vert _{\infty} } \\
&\geq & -\Vert b'\Vert _{\infty} \cdot (t-s) \cdot e ^{ -2(t-s) \Vert b'\Vert _{\infty} } \\
&=& f(\Vert b'\Vert _{\infty} \cdot (t-s)),  
\end{eqnarray*}
with $f(x) = -x\exp(-2x)$. The function $f$ attains its minimum at $x=1/2$, with $f(1/2) = -1/2\exp(-1) > -1$, this prove our claim. The conclusion follows from condition (\ref{7m-3}). 
\end{proof}

Condition (\ref{7m-3}) ensures the existence of the density of the noise $Z_{t}$ for each $t$. This property is then transfered to the solution. 

Let us conclude the existence of the density of the solution to the transport equation.

\begin{theorem}
Let $u(t,x)$ be the solution to the transport equation (\ref{transport}). Assume that $u_{0} \in C^{1}(\mathbb{R})$ such that there exists $C>0$ with $ (u_{0}'(x)) ^ {2} \geq C$ for every $x\in \mathbb{R}$.
Then, for every $t\in [0,T]$ and for every $x\in \mathbb{R}$, the random variable $u(t,x)$ is Malliavin differentiable. Moreover $u(t,x)$ admits a density with respect to the Lebesgue measure.
\end{theorem}
\begin{proof}
	By formula (\ref{19a-1}), $u(t,x)= u_{0} (Y_{0,t}(x))$ and then  we get the Malliavin differentiability of $u(t,x)$ from Proposition \ref{p2}. The  chain rule for the Malliavin derivative  (\ref{chain}) implies
$$D_{\alpha} u(t,x)= u_{0} '( Y_{0,t} (x))  D_{\alpha } Y_{0,t} (x) $$
and then, from the above result and the condition imposed on the initial value $u_{0}$ 
$$\int_{0} ^ {T} (D_{\alpha} u(t,x) ) ^ {2} d\alpha >0$$
almost surely for every $t\in [0,T], x\in \mathbb{R}$. This implies that the random variables $u(t,x)$ admits a density.
\end{proof}

\section{An example: The Hermite process}

 In this section we will give an example of a class of stochastic processes that satisfies the conditions required for the noise $Z$ in (\ref{transport}). Recall that we assumed that the noise $Z$ is an adapted square integrable process, with zero quadratic variation in the Russo-Vallois sense, Malliavin differentiable, and satisfies (\ref{6m-2}) and (\ref{7m-3}).

The class of processes we consider is those of Hermite processes. The Hermite process of order $q\geq 1$ lives in the Wiener chaos of order $q$, and it is defined as a multiple stochastic integral with respect to the standard Brownian motion. Its representation is related to the Wiener integral representation of the fractional Brownian motion. We recall that the fractional Brownian process $(B_{t}^{H})_{t\in \lbrack 0,1]}$
  with Hurst parameter $H\in (0,1)$ can be written as
  \begin{equation}\label{2a-1}
    B_{t}^{H}=\int_{0}^{t} K^{H}(t,s) \, {\rm d}W_{s},\quad t\in \lbrack 0,1]
  \end{equation}
  where $(W_{t},t\in \lbrack 0,T])$ is a standard Wiener process, and if $H>\frac{1}{2}$, the
  kernel $K^{H}\left( t,s\right)$ has the expression 
  $$ K^{H}\left( t,s\right) = c_{H}s^{1/2-H}
  \int_{s}^{t}(u-s)^{H-3/2}u^{H-1/2} \, {\rm d}u$$ where $t>s$ and
  $c_{H}=\left( \frac{H(2H-1)}{\beta (2-2H,H-1/2)}\right) ^{1/2}$ and
  $\beta (\cdot ,\cdot )$ is the Beta function. For $t>s$, the
  kernel's derivative is
$$\frac{\partial K^{H}}{\partial
    t}(t,s)=c_{H}\left( \frac{s}{t}\right)
  ^{1/2-H}(t-s)^{H-3/2}.
$$

  We denote by $(Z_{t}^{(q,H)})_{t\in \lbrack 0,T]}$ the Hermite process with \textit{self-similarity parameter} $H\in \left(1/2,1\right) $. For $t\in [0,T]$ it is given by 
\begin{equation} \label{z1}
\begin{split}
	Z_{t}^{(q,H)} =& d(H) \int_{0}^{t}\ldots\int_{0}^{t} {\rm d}W_{y_{1}} \ldots {\rm d}W_{y_{q}} \ \cdot 
\\
&  \left( \int_{y_{1}\vee\ldots\vee y_{q}}^{t} \partial_{1} K^{H^{\prime}}(u,y_{1}) \ldots \partial_{1} K^{H^{\prime}}(u,y_{q}) {\rm d}u\right),
\end{split}
\end{equation}
where $K^{H^{\prime}}$ is the usual kernel of the fractional
    Brownian motion  that appears in (\ref{2a-1}) and
    \begin{equation*}
      H^{\prime}=1+\frac{H-1}{q}\Longleftrightarrow(2H^{\prime}-2)q=2H-2.
      \label{H'}
    \end{equation*}

The covariance of
  $Z^{\left( q,H\right) }$ is identical to that of fBm, namely
  \begin{equation}\label{26a-2}
   \mathbb{E}\left[ Z_{s}^{\left( q,H\right) }Z_{t}^{\left( q,H\right) }\right]
    =\frac{1}{2}(t^{2H}+s^{2H}-|t-s|^{2H}).
  \end{equation}
  The constant $d(H)$ is chosen to have variance equal to 1. 

  The Hermite process $Z^{(q,H)}$ is $H$-self-similar and it has
    stationary increments,  the mean square of the increment is given by
    \begin{eqnarray}
      \mathbb{E}\left[ \left\vert Z_{t}^{(q,H)}-Z_{s}^{(q,H)}\right\vert ^{2}
      \right] =|t-s|^{2H};
      \label{canonmetric}
    \end{eqnarray}
    as a consequence, using the self-similarity and the stationarity of the increments of $ Z ^ {H}$,   it follows from
    Kolmogorov's continuity criterion (see theorem 2.2.3 in \cite{Oks})  that $Z^{(q,H)}$ has
    H\"{o}lder-continuous paths of any exponent $\delta <H$. For $q=1$, $Z^{(1,H)}$ is standard fBm with Hurst parameter
    $H$, while for $q\geq 2$ the Hermite process is not Gaussian. In
    the case $q=2$ this stochastic process is known as the
      Rosenblatt process.

We will use the notation $Z^{H}:= Z ^{(q,H)}.$ Also, denote by $L^{H}$ the kernel of the Hermite process
\begin{equation}
\label{26a-3}L^{H}_{t}(y_{1},\ldots , y_{q})=1_{(y_{1}\vee\ldots\vee y_{q}\leq t)}\int_{y_{1}\vee\ldots\vee y_{q}}^{t} \partial_{1}
        K^{H^{\prime}}(u,y_{1}) \ldots \partial_{1}
        K^{H^{\prime}}(u,y_{q}) {\rm d}u.
\end{equation}
We can write $ Z ^{H}_{t}= I_{q} (L ^ {H})$ with $I_{q}$ being the multiple integral of order $q$ with respect to the Wiener process $W$. We refer the reader to the manuscript \cite{Tudor2013} and references there in for a deeper discussion on Hermite processes and other self-similar processes.

It is immediate to see that $Z ^ {H}$ has zero quadratic variation as defined in Section 2.
\begin{lemma}
$Z^{H}$ is a zero quadratic variation process.

\end{lemma}
\begin{proof}
	Indeed, from (\ref{canonmetric}),
$$\frac{1}{\varepsilon} \mathbb{E}\int_{0} ^{t} (Z^{H}_{s+\varepsilon} -Z^{H} _{s}) ^{2}ds =t \varepsilon ^{2H-1}\to _{\varepsilon \to 0}0.$$
\end{proof}

We will show that the process $Z^ {H}$ satisfies the assumptions imposed throughout the paper for the noise $Z$ appearing in (\ref{transport}). Since it is defined as a multiple integral, the random variable $Z ^ {H}_{t}$ is clearly Malliavin differentiable for every $t\in [0,T]$.

\begin{lemma}
The Hermite process satisfies (\ref{6m-2}) and (\ref{7m-3}).
\end{lemma}
\begin{proof}
	Clearly, using (\ref{26a-2}) 
$$\sup_{t\in [0,T] }\mathbb{E}\vert Z^{H}_{t} \vert ^{2}= T ^{2H}<\infty.$$
Also, with $L ^ {H}$ given by (\ref{26a-3}),
$$D_{\alpha} Z^{H} _{t} = q I_{q-1} (L^{H}_{t} (\cdot, \alpha )),$$
the properties of multiple Wiener-It\^o integrals gives
\begin{equation*}
\mathbb{E} \int_{0} ^{T} (D_{\alpha } Z^{H}_{t} ) ^{2} d\alpha = q \mathbb{E} [I_{q} (L^{H}_{t})]^{2} = q \mathbb{E}( Z^{H}_{t}) ^{2} = q t^{2H}.
\end{equation*}
This implies 
$$\sup_{t\in [0,T]} \mathbb{E} \int_{0} ^{T} (D_{\alpha } Z^{H}_{t} ) ^{2} d\alpha <\infty.$$

The condition (\ref{7m-3}) is satisfied because $Z^{H}_{t}$ is a multiple integral of order $q$. A classical result by Shikegawa \cite{Shige} (see also \cite{No1}, Corollary 5.2) says that, if $q\geq 1$ is an integer, and $f$ a  of $L^{2}(\mathbb{R}^{q})$ with $\Vert f\Vert _{ L^{2}(\mathbb{R}^{q})}\not=0$ , then the $q$-multiple Wiener-It\^o integral of $f$ has a density and satisfies (\ref{7m-3}). 
\end{proof}

\begin{remark}
In the case $q=1$, that is, $Z ^ {H}$ is a fractional Brownian motion, a deeper analysis of the density of the solution to (\ref{transport}) can  be done. In particular, it is possible to prove Gaussian bounds for the density, see \cite{OT}.
\end{remark}

%For acknowledgements section, please don't number the section, please begin it with \section*{Acknowledgements}

% \section*{Acknowledgments} 

%We would like to thank you for \textbf{following
%the instructions above} very closely in advance. It will definitely save us lot of time and expedite the process of your paper's publication.

% You may incorporate your references as follows in your main tex file.
% Using BibTex is not recommended but can be handled.

\medskip
% The data information below will be filled by AIMS editorial staff
Received xxxx 20xx; revised xxxx 20xx.
\medskip

\end{document}